\documentclass[12pt,reqno]{amsart}
\usepackage{amsmath,a4wide}
\usepackage{amssymb}
\usepackage{esint}
\usepackage{amsthm}
\usepackage{array}
\usepackage{xy}
\usepackage{fontenc}
\usepackage{color}

\newtheorem{thm}{Theorem}[section]
\newtheorem{lemma}[thm]{Lemma}

\theoremstyle{definition}
\newtheorem{defi}{Definition}[section]

\def \de{\partial}

\def \rtri{\mathbb{R}^3}

\def \ep{\varepsilon}
\def \eps{\varepsilon}

\def \d{\mathrm{d}}

\def \n{\mathbf{n}}

\def \u{\mathbf{u}}
\def \x{\mathbf{x}}
\def \y{\mathbf{y}}
\def \ve{\mathbf{v}}

\def \z{\mathbf{z}}
\def \A{\mathbf{A}}
\def \B{\mathbf{B}}
\def \C{\mathbf{C}}
\def \de{\mathbf{D}}
\def \E{\mathbf{E}}

\def \G{\mathbf{G}}

\def \R{\mathbb{R}}
\def \S{\mathbf{S}}
\def \te{\mathbf{T}}
\def \U{\mathbf{U}}

\def \sil{\rightarrow}
\def \sl{\rightharpoonup}
\def\weaklyto{\rightharpoonup}

\renewcommand{\div}{{\rm div}\,}

\newcommand{\abs}[1]{\ensuremath{\left| #1 \right|}}
\newcommand{\norm}[1]{\ensuremath{\left\| #1 \right\|}}

\newcommand{\ds}{\displaystyle}
\newcommand{\beq}[1]{\begin{equation} \label{#1}\ds}
\newcommand{\eeq}{\end{equation}}
\newcommand{\bml}[1]{\beq{#1} \begin{array}{c}\ds}
\newcommand{\eml}{\end{array}\eeq}

\catcode`\@=11 \@addtoreset{equation}{section}

\catcode`\@=12

\newcommand{\lr}[1]{\left( #1 \right)} 
\newcommand{\pd}[2]{{\partial_{#2} #1}} 
\def\grad{\nabla} 
\renewcommand{\t}[1]{\ensuremath{\mathbf{#1}}} 
\def \be{\B(\ve,\te)}
\def \zero{\boldsymbol{0}}
\def \er{\R}
\def \en{\mathbb{N}}
\def\sym{\mathrm{sym}}
\newcommand{\el}[1]{L^{#1}}
\newcommand{\elom}[1]{L^{#1}(\Omega)}
\newcommand{\we}[2]{W^{#1,#2}}
\newcommand{\eq}[1]{\[ \begin{split} #1 \end{split} \]}
\def\eldiv{\el{2}_{0,\div}}
\def\elsym{\el{2}_{\mathrm{sym}}}
\def\wediv{\we{1}{p}_{0,\div}}
\def\wesym{\we{1}{q}_{\mathrm{sym}}}
\def\ce{\mathcal{C}}
\def\cediv{\ce^\infty_{0, \div}\lr{\Omega}}
\newcommand{\bochnerlx}[2]{\ensuremath{\el{#1} \left(0,T; {#2} \right) }} 
\newcommand{\bochnerll}[2]{\bochnerlx{#1}{\elom{#2}}} 

\let\phi=\phiold 
\def\phi{\varphi}
\def\vpsi{\boldsymbol{\psi}}

\def\vphi{\boldsymbol{\phi}}

\newcommand{\duality}[2]{\left\langle #1, #2 \right\rangle}
\newcommand{\norml}[2]{\left\Vert #1 \right\Vert_{#2}} 

\def\dx{\,\d\x}
\def\dt{\,\d t}
\def\xim{\xi^m}
\def\tem{\te^m}
\def\temn{\te^{m,n}}
\def\dem{\de^m}
\def\vem{\ve^m}
\def\vemn{\ve^{m,n}}
\def\bem{\t{B}\lr{\vem, \tem}}
\def\unitball{\mathcal{B}_1}

\begin{document}
\title[On the regularized model of viscoelastic fluid] {On the global existence for a regularized model of viscoelastic non-Newtonian fluid}
\author{Ond\v rej Kreml, Milan Pokorn\'y and Pavel \v Salom} 
\thanks{O.K. acknowledges the support of the GA\v CR (Czech Science Foundation) project GA13-00522S in the general framework of RVO: 67985840.}


\begin{abstract}
We study the generalized Oldroyd model with viscosity depending on the shear stress behaving like $\mu(\de) \sim \abs{\de}^{p-2}$ ($p>\frac 65$) regularized by a nonlinear stress diffusion. Using the Lipschitz truncation method we are able to prove global existence of weak solution to the corresponding system of partial differential equations.
\end{abstract}

\maketitle

\section{Introduction}\label{s:intro}
The well known Oldroyd model describing the flow of incompressible viscoelastic fluid consists of the following system of partial differential equations
\begin{equation}\label{eq:classic Oldroyd}
\begin{split}
\div \ve &= 0, \\
\pd{\ve}{t} + \div\lr{\ve \otimes \ve} &= -\grad \pi + \mu\Delta\ve + \div\te, \\
\te + \pd{\te}{t} + \ve\cdot\grad\te &= 2\mu_0\de + \t{W}\te - \te\t{W} + a\lr{\de\te + \te\de}. 
\end{split}
\end{equation}
Here the unknowns are the velocity vector $\ve$, the pressure $\pi$ and the symmetric extra stress tensor $\te$. The tensor $\de$ denotes the symmetric part of the velocity gradient $\de = \de(\ve) = \frac 12 (\grad\ve + (\grad\ve)^T)$ and $\t{W}$ denotes its skew--symmetric part $\t{W} = \t{W}(\ve) = \frac 12 (\grad\ve - (\grad\ve)^T)$, $\mu$ and $\mu_0$ are positive constants and $a \in [-1,1]$ is a real parameter. Special choices $a = -1,0,1$ yield respectively the lower convected (Oldroyd A), corotational and upper convected (Oldroyd B) models.

With the exception of the work of Lions and Masmoudi \cite{LiMa}, where the authors proved global existence of weak solutions for the corotational model, the global existence theory for the Oldroyd models is still an open problem. Existence of weak solutions to \eqref{eq:classic Oldroyd} for general $a$ is proved only under some smallness assumptions, either on the time interval or the initial data (see e.g.  \cite{FaHiZi}, \cite{KrPo}, \cite{LiLiZh}, \cite{LiWa}).

It is well known that some fluids as e.g. the blood exhibit both the viscoelastic and shear-thinning behavior. Therefore it is important to consider models which can describe these properties. 
In this paper we propose a generalized, and regularized, version of the Oldroyd system \eqref{eq:classic Oldroyd}. Namely, instead of a constant viscosity coefficient $\mu$ in \eqref{eq:classic Oldroyd}$_2$ we introduce shear dependent viscosity $\mu(\de)$ with properties specified later. This enables the model to describe better the shear thinning behavior of the fluid (or shear thickening, if needed). As even the model with constant viscosity is (except for a special case discussed above) not known to posses a weak solution, the least it can be expected for a more complex (and less regular) model. Hence we regularize  equation \eqref{eq:classic Oldroyd}$_3$ for the extra stress by introducing a (nonlinear) stress diffusion. Denoting
\begin{equation}\label{eq:B}
\be := \t{W}\te - \te\t{W} + a\lr{\de\te + \te\de}
\end{equation}
the system we study is the following
\begin{equation}\label{eq:strong formulation}
\begin{split}
\div\ve &= 0, \\
\pd{\ve}{t} + \div\lr{\ve \otimes \ve} + \grad \pi - \div(\mu(\de)\de) &= \div\te, \\
\pd{\te}{t} + \ve\cdot\grad\te - \eps \div(\gamma(\grad\te)\grad\te) + \te &= 2\mu_0\de + \be.
\end{split}
\end{equation}
Here $\ep$ is a positive  constant and the properties of functions $\mu(\de)$ and $\gamma(\nabla\te)$ are stated later.

We consider our system \eqref{eq:strong formulation} on a space-time cylinder $\Omega \times [0,T)$ where $\Omega \subset \rtri$ is a bounded domain with Lipschitz boundary and we add the initial conditions
\begin{equation}\label{eq:IC}
\begin{split}
\ve\lr{\x,0} &= \ve_0\lr{\x} \quad \textrm{in $\Omega$}, \\
\te\lr{\x,0} &= \te_0\lr{\x} \quad \textrm{in $\Omega$}
\end{split}
\end{equation}
and the boundary conditions
\begin{equation}\label{eq:BC}
\begin{split}
\ve &= \zero \quad \textrm{on $\partial\Omega \times [0,T)$}, \\
\frac{\partial\te}{\partial \n} &= \zero \quad \textrm{on $\partial\Omega \times [0,T)$}.
\end{split}
\end{equation}
We assume that the function $\mu: \R^{3\times 3} \sil \R^+$ satisfies the following conditions. For some $p>1$
\begin{itemize}
\item[(i)] $\mu\lr{\de}\de$ is $p$-coercive, i.e.
\begin{equation}\label{assum: p-coercivity for mu}
\exists c>0\ \exists\phi_1\in\el{1}\lr{\Omega\times \lr{0,T}}\ \forall\de\in\er^{3\times 3}_\sym: \quad \mu\lr{\de}|\de|^2 \geq c|\de|^p-\phi_1,
\end{equation}
\item[(ii)] $\mu\lr{\de}\de$ has $(p-1)$-growth, i.e.
\begin{equation}\label{assum: p-1 growth for mu}
\exists c>0\ \exists\phi_2\in\el{p'}\left(\Omega\times \lr{0,T}\right)\ \forall\de\in\er^{3\times 3}_\sym: \quad \mu(\de)|\de| \leq c|\de|^{p-1}+\phi_2,
\end{equation}
\item[(iii)] $\mu\lr{\de}\de$ is strictly monotone, i.e.
\begin{equation}\label{assum: monotonicity for mu}
\forall\de_1, \de_2\in\er^{3\times 3}_\sym, \, \de_1\neq\de_2
: \quad \left(\mu(\de_1)\de_1 - \mu(\de_2)\de_2\right):\left(\de_1-\de_2\right) > 0,
\end{equation}
\end{itemize}
and similarly $\gamma: \R^{3\times 3\times 3} \sil \R^+$ satisfies for some $q>1$
\begin{itemize}
\item[(iv)] $\gamma\lr{\grad\te}\grad\te$ is $q$-coercive, i.e.
\begin{equation}\label{assum: q-coercivity for gamma}
\exists c>0\ \exists\phi_3\in\el{1}\lr{ \Omega\times \lr{0,T} }\ \forall\te\in\er^{3\times 3}_\sym: \quad \gamma(\grad\te)|\grad\te|^2 \geq c|\grad\te|^q-\phi_3,
\end{equation}
\item [(v)] $\gamma\lr{\grad\te}\grad\te$ has $(q-1)$-growth, i.e.
\begin{equation}\label{assum: q-1 growth for gamma}
\exists c>0\ \exists\phi_4\in\el{q'}\lr{ \Omega\times \lr{0,T} }\ \forall\te\in\er^{3\times 3}_\sym: \quad \gamma(\grad\te)|\grad\te| \leq c|\grad\te|^{q-1}+\phi_4,
\end{equation}
\item[(vi)] $\gamma\lr{\grad\te}\grad\te$ is monotone, i.e.
\begin{equation}\label{assum: monotonicity for gamma}
\forall\te_1, \te_2\in\er^{3\times 3}_\sym 
: \quad \left(\gamma(\grad\te_1)\grad\te_1-\gamma(\grad\te_2)\grad\te_2\right):\left(\grad\te_1-\grad\te_2\right) \geq 0.
\end{equation}
\end{itemize}

This model was first introduced in \cite{KrTh} where existence of weak solutions for $p>\frac 85$ and $q$ sufficiently large was proven for the problem with either periodic boundary conditions or complete slip boundary conditions for the velocity $\ve$. The proof in \cite{KrTh} is based on the $L^\infty$ test functions technique developed by Frehse, M\'alek and Steinhauer in \cite{FrMaSt}. The case of Dirichlet boundary conditions for problem \eqref{eq:strong formulation} was studied in \cite{SaTh}, where existence of weak solutions is proved for $p > \frac 65$ and $q > 2p'$. The proof is based on the Lipschitz truncation method and the construction of local pressure from \cite{DiRuWo}. Moreover, in both cases, additional lower order nonlinear term was used in order to obtain suitable a-priori estimates. In this paper we further improve the condition on $q$,  under which the existence of weak solutions is proved, remove the additional lower-order term and based on the recent result for the Lipschitz truncation to the solenoidal functions (see \cite{BrDiSc}) we also significantly shorten the proof.

Similar model was studied in \cite{CoKl} where the authors consider classical Oldroyd-B model with constant viscosity $\mu$ and linear stress diffusion and prove global regularity of solutions in 2D.

In the whole text we denote vectors by small bold letters and tensors by capital bold letters. We introduce the following function spaces.
\eq{
\cediv &:= \left\{ \ve:\er^3\to\er^3 \mid \ve\in\ce_0^\infty\lr{\Omega}, \div \ve = 0 \right\}, \\
\eldiv(\Omega) &:= \overline {\left\{ \ve:\er^3\to\er^3 \mid \ve\in\cediv \right\}}^{\norm{\cdot}_{\el{2}\lr{\Omega}}}, \\
\wediv(\Omega) &:= \overline {\left\{ \ve:\er^3\to\er^3 \mid \ve\in\cediv \right\}}^{\norm{\cdot}_{\we{1}{p}\lr{\Omega}}}, \\
\elsym(\Omega) &:= \overline {\left\{ \te:\er^3\to\er^{3\times 3} \mid \te\in\ce^\infty\lr{\overline{\Omega}}, \textrm{$\te$ symmetric} \right\}}^{\norm{\cdot}_{\el{2}\lr{\Omega}}}, \\
\wesym(\Omega) &:= \overline {\left\{ \te:\er^3\to\er^{3\times 3} \mid \te\in\ce^\infty\lr{\overline{\Omega}}, \textrm{$\te$ symmetric} \right\}}^{\norm{\cdot}_{\we{1}{q}\lr{\Omega}}}.
}
Moreover, we denote $X^*$ the dual space to $X$ and by $\duality{T}{\phi}_{k,r}$ we mean duality between spaces $\we{k}{r}(\Omega)$ and $\lr{\we{k}{r}(\Omega)}^*$, similarly for duality between Sobolev space of solenoidal functions with zero trace (or symmetric tensors) and its dual we use $\duality{T}{\phi}_{k,r,\div}$ (or $\duality{T}{\phi}_{k,r,\sym}$ respectively). For $t > 0$ we denote $Q_t = \Omega\times (0,t)$ the space--time cylinder. For $s \in [1,\infty]$ we denote $s'$ its dual exponent, i.e. $\frac{1}{s} + \frac{1}{s'} = 1$.

Before defining weak solution of \eqref{eq:strong formulation} we denote 
\begin{equation}\label{eq:A}
\A(\te,\vpsi) = -\te\vpsi + \lr{\te\vpsi}^T + a\lr{\te\vpsi + \lr{\te\vpsi}^T} 
\end{equation}
and observe that integrating by parts and using the boundary condition \eqref{eq:BC}$_1$ it holds
\begin{equation}\label{eq:B to A}
\int_\Omega \be:\vpsi \dx\dt = - \int_\Omega \ve\cdot\div\A(\te,\vpsi) \dx\dt
\end{equation}
for all $\vpsi \in \ce^\infty\lr{\overline\Omega}$.
Note moreover that in our case $\vpsi$ and $\te$ are symmetric, therefore
$$
\A(\te,\vpsi) = \vpsi\te -\te\vpsi  + a\lr{\te\vpsi + \vpsi \te}. 
$$
\begin{defi}\label{def: definition of weak solution}
Let $\ve_0 \in \eldiv(\Omega)$, $\te_0 \in \elsym(\Omega)$ and let $\eps$, $\mu_0$, $T$ be positive constants. Let $\mu:\er^{3\times 3} \to \er^+$ be a continuous function satisfying \eqref{assum: p-coercivity for mu}--\eqref{assum: monotonicity for mu} with some $p > \frac 65$ and let $\gamma:\er^{3\times 3\times 3} \to \er^+$ be a continuous function satisfying \eqref{assum: q-coercivity for gamma}--\eqref{assum: monotonicity for gamma} with some $q > 1$. We say that a couple $(\ve, \te)$ is a \textit{weak solution of system \eqref{eq:strong formulation}} with initial conditions \eqref{eq:IC} and boundary conditions \eqref{eq:BC} if
\eq{
\ve &\in \bochnerlx{\infty}{\eldiv(\Omega)} \cap \bochnerlx{p}{\wediv(\Omega)}, \\
\te &\in \bochnerlx{\infty}{\elsym(\Omega)} \cap \bochnerlx{q}{\wesym(\Omega)}, \\
\pd{\ve}{t} &\in \bochnerlx{\sigma}{\lr{\we{1}{\sigma'}_{0,\div}(\Omega)}^*}, \quad \textrm{for some} \quad 1 \leq \sigma \leq \frac{5}{6}p, \\
\pd{\te}{t} &\in \bochnerlx{q'}{\lr{\wesym(\Omega)}^*}
}
and for almost all $t\in\lr{0,T}$ it holds 
\begin{equation}\label{eq: weak formulation - 1st equation}
\duality{\pd{\ve}{t}\lr{t}}{\vphi}_{1,\sigma',\div} 
+ \int_{\Omega} \Big(\mu\lr{\de\lr{t}}\de\lr{t} - \lr{\ve\otimes\ve}\Big) : \grad\vphi \dx = \int_{\Omega} \div \te\lr{t} \cdot \vphi \dx
\end{equation}
for all  $\vphi \in \cediv$,
\begin{equation}\label{eq: weak formulation - 2nd equation}
\begin{split}
\duality{\pd{\te}{t}\lr{t}}{\vpsi}_{1,q,\sym} + \int_{\Omega} \lr{\ve\lr{t}\cdot\grad\te\lr{t}} : \vpsi \dx \\ + \eps\int_{\Omega} \gamma\lr{\grad\te\lr{t}}\grad\te\lr{t} : \grad\vpsi \dx + \int_\Omega \te : \vpsi \dx \\
 = 
2\mu_0\int_\Omega \de\lr{t}: \vpsi\dx 
- \int_\Omega \ve\lr{t}\cdot\div\A\lr{\te\lr{t},\vpsi} \dx 
\end{split}
\end{equation}
for all $\vpsi \in \ce^\infty\lr{\overline{\Omega}}$, $\vpsi$ symmetric,
\begin{equation}\label{eq: weak formulation - 1st initial condition}
\begin{split}
\forall \vphi\in\cediv: \quad &\lim_{t\to0_+} \int_\Omega \ve\lr{t}\cdot\vphi\dx = \int_\Omega \ve_0 \cdot \vphi \dx
\end{split}
\end{equation}
\begin{equation}\label{eq: weak formulation - 2nd initial condition}
\begin{split}
\forall \vpsi\in\ce^\infty\lr{\overline{\Omega}}: \quad &\lim_{t\to0_+} \int_\Omega \te\lr{t} : \vpsi\dx = \int_\Omega \te_0 : \vpsi \dx
\end{split}
\end{equation}
and \eqref{eq:BC}$_1$ is fulfilled for almost all $t\in\lr{0,T}$ in the sense of traces.
\end{defi}

We are now ready to state the main theorem of this paper.
\begin{thm}\label{t:main}
Let $\Omega \subset \er^3$ be a bounded domain with Lipschitz boundary. Let $\ve_0 \in \eldiv(\Omega)$, $\te_0 \in \elsym(\Omega)$ and let $\eps$,  $\mu_0$, $T$ be given positive constants. Let $\mu:\er^{3\times 3} \to \er^+$ be a continuous function satisfying \eqref{assum: p-coercivity for mu}--\eqref{assum: monotonicity for mu} and let $\gamma:\er^{3\times 3\times 3} \to \er^+$ be a continuous function satisfying \eqref{assum: q-coercivity for gamma}--\eqref{assum: monotonicity for gamma}. Moreover, let
\begin{equation}\label{assum: assumption on exponents in Main theorem} 
\begin{array} {rl}
\displaystyle \frac{6}{5}< p\leq 2, &\qquad q \geq 4, \\
2<p, &\displaystyle \qquad q > \frac{2p}{p-1}.
\end{array}
\end{equation}
Then there exists a weak solution to system \eqref{eq:strong formulation} with initial conditions \eqref{eq:IC} and boundary conditions \eqref{eq:BC}.
\end{thm}


The rest of the paper is devoted to the proof of Theorem \ref{t:main}. The proof is organized as follows. In Section \ref{s:approx} we present the approximative system, show the existence of weak solutions to it and derive a-priori estimates. In Section \ref{s:limit1} we perform the first limiting procedure, pass to the limit in most of the terms in the equations and identify the main problem which we have to solve. In Section \ref{s:lip} we recall some recently proved properties of the Lipschitz truncation operator which is a key tool in the final step of the proof. Then we complete the proof.

\section{Approximation and a priori estimates}\label{s:approx}

Let $\xi \in \ce^\infty \lr{[0,\infty)}$ be a non-increasing function such that $\xi \equiv 1$ in $[0,1]$, $\xi \equiv 0$ in $[2,\infty)$ and $0 \geq \xi' \geq -2$. For $m \in \mathbb{N}$ we define
\eq{
\xim \lr{s} := \xi \lr{\frac{s}{m}}, \quad s \in [0,\infty).
}
We consider the approximative system (denoting $\dem = \de(\vem)$, $\tem_M(t) = \frac{1}{|\Omega|} \int_{\Omega} \tem(t,\cdot)\dx$)
\begin{equation}\label{eq:approx}
\begin{split}
\div\vem &= 0, \\
\pd{\vem}{t} + \div\lr{\vem \otimes \vem \xim\lr{\abs{\vem}}} + \grad \pi^m &- \div(\mu(\dem)\dem) = \div\tem, \\
\pd{\tem}{t} + \vem\cdot\grad\tem - \eps \div(\gamma(\grad\tem)\grad\tem) & \\
+\frac 1 m \abs{\tem_M}^{q-2}\tem + \tem &= 2\mu_0\dem + \bem
\end{split}
\end{equation}
with initial conditions \eqref{eq:IC} and boundary conditions \eqref{eq:BC}. 
\begin{defi}\label{d:approx}
By a {\em weak solution to  system \eqref{eq:approx}} we mean a couple $(\vem,\tem)$ such that
\eq{
\vem &\in \ce\lr{[0,T];\eldiv(\Omega)} \cap \bochnerlx{p}{\wediv(\Omega)}, \\
\tem &\in \ce\lr{[0,T];\elsym(\Omega)} \cap \bochnerlx{q}{\wesym(\Omega)}, \\
\pd{\vem}{t} &\in \bochnerlx{p'}{\lr{\we{1}{p}_{0,\div}(\Omega)}^*}, \\
\pd{\tem}{t} &\in \bochnerlx{q'}{\lr{\we{1}{q}_{\sym}(\Omega)}^*}
}
satisfying   
\begin{equation}\label{eq: approx weak formulation - 1st equation}
\begin{split}
\int_0^T \duality{\pd{\vem}{t}}{\vphi}_{1,p,\div} \dt
- \int_0^T \int_{\Omega} \lr{\vem\otimes\vem\xim\lr{\abs{\vem}}} : \grad\vphi \dx\dt \\
+ \int_0^T \int_{\Omega} \mu\lr{\dem}\dem : \grad\vphi \dx\dt = \int_0^T \int_{\Omega} \div \tem \cdot \vphi \dx\dt
\end{split}
\end{equation}
for all  $\vphi \in \bochnerlx{p}{\wediv(\Omega)}$,
\begin{equation}\label{eq: approx weak formulation - 2nd equation}
\begin{split}
\int_0^T \duality{\pd{\tem}{t}}{\vpsi}_{1,q} \dt + \int_0^T\int_{\Omega} \lr{\vem\cdot\grad\tem} : \vpsi \dx\dt  \\ + \eps\int_0^T\int_{\Omega} \gamma\lr{\grad\tem}\grad\tem : \grad\vpsi \dx\dt + \int_0^T\int_\Omega\lr{\frac 1 m \abs{\tem_M}^{q-2}\tem + \tem} : \vpsi \dx\dt \\
 = 
2\mu_0\int_0^T\int_\Omega \dem: \vpsi\dx\dt 
- \int_0^T\int_\Omega \vem\cdot\div\A(\tem,\vpsi)\dx\dt
\end{split}
\end{equation}
for all $\vpsi \in \bochnerlx{q}{\wesym(\Omega)}$.
\end{defi}

\begin{thm}\label{t:approx}
Let $\Omega \subset \er^3$ be a bounded domain with Lipschitz boundary. Let $\ve_0 \in \eldiv(\Omega)$, $\te_0 \in \elsym(\Omega)$ and let $\eps$, $\mu_0$, $T$ be positive constants, $m \in \en$. Let $\mu:\er^{3\times 3} \to \er^+$ be a continuous function satisfying \eqref{assum: p-coercivity for mu}--\eqref{assum: monotonicity for mu} and let $\gamma:\er^{3\times 3\times 3} \to \er^+$ be a continuous function satisfying \eqref{assum: q-coercivity for gamma}--\eqref{assum: monotonicity for gamma}. Moreover, let \eqref{assum: assumption on exponents in Main theorem} hold. Then there exists a weak solution of system \eqref{eq:approx} with initial conditions \eqref{eq:IC} and boundary conditions \eqref{eq:BC}.
\end{thm}
\begin{proof}[Proof of Theorem \ref{t:approx}]
The existence is proved using the standard Galerkin method. We look for the approximations in the form 
$$
\begin{array}{c}
\vemn(t,\x) = \sum_{i=1}^n c_i^n(t) {\mathbf w}^i(\x), \\
\temn(t,\x) = \sum_{i=1}^n d_i^n(t) {\mathbf W}^i(\x),
\end{array}
$$
where $\{{\mathbf w}^i(x)\}_{i=1}^\infty$ is an orthogonal system in $W^{1,2}_{0,\div}(\Omega)$, orthonormal in $L^2(\Omega)$ formed by smooth functions and $\{{\mathbf W}^i(x)\}_{i=1}^\infty$ is an orthogonal system in $W^{1,2}_{\mathrm{sym}}(\Omega)$, orthonormal in $L^2(\Omega)$ formed by smooth functions  such that  ${\mathbf W}^1(\x) = \frac{1}{\sqrt{3|\Omega|}}{\mathbf 1}_{\Omega}(\x)\mathbf{I}$, where ${\mathbf 1}_{\Omega}(\x)$ is the characteristic function of the domain $\Omega$ and $\mathbf{I}$ denotes the 3x3 identity matrix. The definition of the Galerkin approximation is standard. Local in time existence for solution to the Galerkin approximation is a direct consequence of the theory of the systems of ODEs. The solutions are global as soon as we are able to prove certain a priori estimates. If we ``test" the Galerkin approximation for $\vemn$ by $2\mu_0 \vemn$ and the approximation for $\temn$ by $\temn$ and add the formulation for $\temn$ ``tested" by ${\mathbf W}^1$ (i.e. integrated over $\Omega$) and multiplied by $\temn_M$, we get
\begin{equation} \label{ineq:Galerkin}
\begin{split}
& \frac 12 \frac{\textrm{d}}{\textrm{dt}} \Big(2\mu_0 \norml{\vemn}{2}^2 +  \norml{\temn}{2}^2 +   \sqrt{\frac{\abs{\Omega}}{3}}\abs{\temn_M}^2\Big) \\
& +  2\mu_0 c \norml{\grad\vemn}{p}^p + \eps c  \norml{\grad\temn}{q}^q + \frac{c}{m}  \abs{\temn_M}^{q} \\
& \leq C(t) + c \int_\Omega |\vemn||\nabla \temn|(|\temn| + |\temn_M|)\dx.
\end{split}
\end{equation}
Under assumptions of the main theorem, it is an easy matter to estimate the integral on the right-hand side and we get control of norms coming from the left-hand side. The procedure is similar to estimates in Theorem \ref{t:estim}, only slightly easier. Note, however, that the control depends on $m$; indeed, it is independent of $n$. Next, by duality argument, we also prove  estimates of the time derivatives of $\vemn$ and $\temn$ as stated in Definition \ref{d:approx} and using the Aubin-Lions lemma we get strong convergence of $\vemn$ and $\temn$ in $L^2((0,T)\times \Omega)$. The proof of Theorem \ref{t:approx} is completed by the standard monotonicity argument (the Minty trick, see also Section \ref{s:limit1}) and the density argument to extend the class of the test functions. 
\end{proof}

\begin{thm}\label{t:estim}
Under the assumptions of Theorem \ref{t:approx}, for any $m \in \en$, we have
the following estimates
\begin{equation}\label{ineq: a priori estimate}
\begin{split}
\norm{\vem}_{\bochnerll{\infty}{2}} &+ \norm{\tem}_{\bochnerll{\infty}{2}} + \norm{\grad\vem}_{\bochnerll{p}{p}} + \norm{\grad\tem}_{\bochnerll{q}{q}} \leq c
\end{split}
\end{equation}
with $c$ independent of $m$.
\end{thm}
\begin{proof}[Proof of Theorem \ref{t:estim}]

For fixed $m$, we proceed similarly as above. However, we cannot use the information from the term $\frac 1m |\tem_M|^{q-2} \tem$, hence the estimates are slightly more complex. Unlike the Galerkin approximation, we keep the estimates separately. We will not subtract two terms, however, they are of lower order and do not cause any troubles.  
We test equation \eqref{eq:approx}$_2$ by $\vem$ and get
\begin{equation} \label{ineq:velocity}
\frac 12\int_0^s \frac{\textrm{d}}{\textrm{dt}} \norml{\vem}{2}^2 \dt +  c \int_0^s \norml{\grad\vem}{p}^p \dt \leq \int_0^s\int_\Omega \phi_1 \dx\dt +\int_0^s \int_\Omega |\vem| |\nabla \tem| \dx\dt.
\end{equation}
Next we
 test equation \eqref{eq:approx}$_3$ by $\tem$ and have 
\begin{equation}\label{ineq:stress}
\begin{split}
&\frac{1}{2} \int_0^s \frac{\textrm{d}}{\textrm{dt}} \norml{\tem}{2}^2 \dt  
+ \eps  \int_0^s \norml{\grad\tem}{q}^q \dt  + \int_0^s \norml{\tem}{2}^2 \dt \\
&\leq \int_0^s\int_\Omega \phi_3 \dx\dt + \int_0^s \int_\Omega |\vem| |\grad \tem| \dx\dt +\int_0^s\int_\Omega |\vem||\tem||\nabla\tem| \dx\dt.
\end{split}
\end{equation}
The information about $\tem$ on the left-hand side is not sufficient, as it is in the second power. We therefore use the same trick as before; we integrate \eqref{eq:approx}$_3$ over $\Omega$ and multiply it by $\tem_M$. Hence
\begin{equation} \label{ineq:mean}
\frac 12  \int_0^s \frac{\textrm{d}}{\textrm{dt}} |\tem_M|^2 \dt + \int_0^s|\tem_M|^2\dt  \leq c \int_0^s |\tem_M| \int_\Omega |\vem||\nabla \tem|\dx\dt. 
\end{equation}
We first take $p\leq 2$. Then we have from \eqref{ineq:mean}
$$
\|\tem_M\|_{L^\infty(0,s)} \leq C\Big(1+ \int_0^s \|\vem\|_{\elom{2}} \norm{\grad\tem}_{\elom{q}}\dt\Big).
$$
Hence, writing the last term in \eqref{ineq:stress} as
\begin{multline*}
\int_0^s \int_\Omega |\vem||\tem| |\nabla \tem|\dx\dt \\ \leq
\int_0^s \int_\Omega |\vem||\tem-\tem_M| |\nabla \tem|\dx\dt +
\int_0^s |\tem_M| \int_\Omega |\vem||\nabla \tem|\dx\dt,
\end{multline*}
we have using Poincar\'e's inequality
\begin{multline*}
\norm{\grad\tem}_{L^{q}(0,s;\elom{q})}^q \leq C\Big(1+ \Big(\int_0^s \|\vem\|_{\elom{2}}^2\dt\Big)^{\frac 12} \Big(\int_0^s\|\grad\tem\|_{\elom{q}}^q\dt\Big)^{\frac 1q}  
\\ +\Big(\int_0^s \|\vem\|_{\elom{2}}^2\dt\Big)^{\frac 12} \Big(\int_0^s\|\grad\tem\|_{\elom{q}}^q\dt\Big)^{\frac 2q} \\
+ \Big(\int_0^s \|\vem\|_{\elom{2}}^2\dt\Big) \Big(\int_0^s\|\grad\tem\|_{\elom{q}}^q\dt\Big)^{\frac 2q}
\Big).
\end{multline*}
Therefore
$$
\norm{\grad\tem}_{\bochnerll{q}{q}} \leq C\Big(1+ \Big(\int_0^s \|\vem\|_{\elom{2}}^2\dt\Big)^{\frac 1{q-2}}\Big) .
$$
Finally, from \eqref{ineq:velocity}
$$
\|\vem(s,\cdot)\|_{L^2(\Omega)}^2 \leq C\Big(1+ \Big(\int_0^s \|\vem\|_{L^2(\Omega)}^2\dt\Big)^{\frac 12 +\frac{1}{q-2}}\Big),
$$
which leads to the estimate by virtue of the Gronwall lemma  provided $q\geq 4$. For $p>2$ we use the $L^p$-norm of the velocity gradient. Proceeding similarly as above we get
$$
\norm{\grad \vem}_{\bochnerll{p}{p}}^p \leq C(1+ (\norm{\grad \vem}_{\bochnerll{p}{p}})^{1+\frac{2}{q-2}})
$$
which gives the required a-priori estimates provided $q> \frac{2p}{p-1}$. 
\end{proof}


\section{Limiting procedure I}\label{s:limit1}

In what follows, we consider only the more interesting case $p\leq 2$. The other case can be proved similarly, we only need to work with different spaces corresponding to the a-priori estimates. 
\begin{lemma}\label{l:conv}
Let $\left\{(\vem,\tem)\right\}_{m=1}^{\infty}$ be a sequence of weak solutions of \eqref{eq:approx}. Then there exists a subsequence (not relabeled) such that
\begin{eqnarray}
\label{conv: final lp - weak convergence of v^m} \vem &\weaklyto& \ve \quad \text{weakly in} \ \bochnerlx{p}{\wediv(\Omega)}, \\
\label{conv: final lp - weak convergence of v^m 2} \vem &\weaklyto^*& \ve \quad \text{weakly$^*$ in} \ \bochnerll{\infty}{2}, \\
\label{conv: final lp - weak convergence of T^m} \tem &\weaklyto& \te \quad \text{weakly in} \ \bochnerlx{q}{\wesym(\Omega)}, \\
\label{conv: final lp - weak convergence of T^m 2} \tem &\weaklyto^*& \te \quad \text{weakly$^*$ in} \ \bochnerll{\infty}{2}, \\
\label{conv: final lp - weak convergence of muD^m} \mu\lr{\dem}\dem &\weaklyto& \S \quad \text{weakly in} \ \bochnerll{p'}{p'}, \\
\label{conv: final lp - weak convergence of gammaT^m} \gamma\lr{\nabla\tem}\nabla\tem &\weaklyto& \U \quad \text{weakly in} \ \bochnerll{q'}{q'}, \\
\label{conv: final lp - weak convergence of v^m in L^5/3} \vem &\weaklyto& \ve \quad \text{weakly in} \ \bochnerll{\frac{5}{3}p}{\frac{5}{3}p},
\end{eqnarray}
\begin{eqnarray}
\label{conv: final lp - weak convergence of time derivative of v^m} \pd{\vem}{t} &\weaklyto& \pd{\ve}{t} \quad \text{weakly in} \ \bochnerlx{\sigma}{\lr{\we{1}{\sigma'}_{0,\div}(\Omega)}^*}, \quad 1\leq\sigma \leq \frac{5}{6}p, \\
\label{conv: final lp - weak convergence of time derivative of T^m} \pd{\tem}{t} &\weaklyto& \pd{\te}{t} \quad \text{weakly in} \ \bochnerlx{q'}{\lr{\wesym(\Omega)}^*}, \\
\label{conv: final lp - strong convergence of v^m 1} \vem &\to& \ve \quad \text{strongly in} \ \bochnerll{p}{r}, \quad 1\leq r < \frac{3p}{3-p},\\
\label{conv: final lp - strong convergence of T^m} \tem &\to& \te \quad \text{strongly in} \ \bochnerll{q}{\overline{r}}, \quad 1\leq \overline{r} \leq \infty,\\
\label{conv: final lp - strong convergence of v^m 2} \vem &\to& \ve \quad \text{strongly in} \ \bochnerll{2\sigma}{2\sigma}, \quad 1\leq 2\sigma < \frac{5}{3}p,\\
\label{conv: final lp - strong convergence of v^m otimes v^m} \vem \otimes \vem\xim\lr{\abs{\vem}} &\to& \ve\otimes\ve \quad \text{strongly in} \ \bochnerll{\sigma}{\sigma}, \quad 1\leq\sigma < \frac{5}{6}p, \\
\label{conv: final l1 - strong convergence of regularization} \frac 1m |\tem_M|^{q-2}\tem  &\to& {\mathbf 0} \quad \text{strongly in} \ \bochnerll{1}{1}.
\end{eqnarray}
\end{lemma}
\begin{proof}[Proof of Lemma \ref{l:conv}]
Convergences \eqref{conv: final lp - weak convergence of v^m}--\eqref{conv: final lp - weak convergence of T^m 2} are direct consequences of a-priori estimates \eqref{ineq: a priori estimate}. Convergences \eqref{conv: final lp - weak convergence of muD^m} and \eqref{conv: final lp - weak convergence of gammaT^m} are achieved combining  a-priori estimate \eqref{ineq: a priori estimate} with \eqref{assum: p-1 growth for mu} and \eqref{assum: q-1 growth for gamma} respectively. Interpolating between $L^\infty(0,T;L^2(\Omega))$ and $L^p(0,T;\wediv(\Omega))$ yields \eqref{conv: final lp - weak convergence of v^m in L^5/3}. 

Next we want to prove \eqref{conv: final lp - weak convergence of time derivative of v^m}. To this aim it is enough to prove a priori bound for the time derivative of $\vem$ in $\bochnerlx{\sigma}{\lr{\we{1}{\sigma'}_{0,\div}}^*}$. We denote 
\eq{
\unitball := \left\{\vphi\in\bochnerlx{\sigma'}{\we{1}{\sigma'}_{0,\div}} , \norm{\vphi}_{\bochnerlx{\sigma'}{\we{1}{\sigma'}_{0,\div}}}\leq 1 \right\}
}
and using \eqref{eq: approx weak formulation - 1st equation} we estimate
\eq{
&\norm{\pd{\vem}{t}}_{\bochnerlx{\sigma}{\lr{\we{1}{\sigma'}_{0,\div}}^*}} = \sup_{\vphi\in\unitball} \abs{ \int_0^T \duality{\pd{\vem}{t}}{\vphi}_{1,\sigma',\div} \dt } \\
&= \sup_{\vpsi\in\unitball} \abs{\int_{Q_T} \lr{ \vem\otimes\vem\xim\lr{\abs{\vem}} - \mu\lr{\dem}\dem - \tem }:\grad\vphi\ \dx\dt} \\
&\leq \sup_{\vpsi\in\unitball} \Bigg( \norm{\vem}_{L^{2\sigma}(Q_T)}^2 + c\norm{\dem}_{L^{(p-1)\sigma}(Q_T)}^{p-1} + \norm{\phi_2}_{L^{\sigma}(Q_T)} + \norm{\tem}_{L^{\sigma}(Q_T)} \Bigg) \norm{\grad\vphi}_{L^{\sigma'}(Q_T)} \\
&\leq c\lr{1 + \norm{\vem}_{\bochnerll{\frac{5}{3}p}{\frac{5}{3}p}}^2 + \norm{\dem}_{\bochnerll{(p-1)\sigma}{(p-1)\sigma}}^{p-1} + \norm{\tem}_{\bochnerll{2}{2}}} \\
&\leq c\lr{1 + \norm{\vem}_{\bochnerll{\infty}{2}}^2 + \norm{\grad\vem}_{\bochnerll{p}{p}}^2 + \norm{\dem}_{\bochnerll{p}{p}}^{p-1} + \norm{\tem}_{\bochnerll{q}{q}}} \\
&\leq c;
}
we used $\sigma \leq \frac{5}{6}p < 2 $ when we estimated $\norm{\phi_2}$ and $\norm{\tem}$ in $\bochnerll{\sigma}{\sigma}$ and $\norm{\dem}$ in $\bochnerll{\lr{p-1}\sigma}{\lr{p-1}\sigma}$. 
Convergence \eqref{conv: final lp - weak convergence of time derivative of v^m} now follows easily. In the same way we derive also \eqref{conv: final lp - weak convergence of time derivative of T^m} using \eqref{eq: approx weak formulation - 2nd equation}. Having estimates for time derivatives, convergences \eqref{conv: final lp - strong convergence of v^m 1} and \eqref{conv: final lp - strong convergence of T^m} are direct consequences of Aubin--Lions lemma. Convergence \eqref{conv: final lp - strong convergence of v^m 2} follows from \eqref{conv: final lp - weak convergence of v^m in L^5/3} and \eqref{conv: final lp - strong convergence of v^m 1} with $r=p$, by interpolation.
Further, \eqref{conv: final lp - strong convergence of v^m 2} implies that $\vem(\x,t) \sil \ve(\x,t)$ a.e. in $Q_T$, which together with the uniform estimate
\eq{
\int_{Q_T} \abs{ \vem\otimes\vem\xim\lr{\abs{\vem}} }^\sigma \dx\dt \leq \int_{Q_T} \abs{\vem}^{2\sigma} \dx\dt \leq c
}
and a combination of Lebesgue's and Vitali's theorem imply the strong convergence \eqref{conv: final lp - strong convergence of v^m otimes v^m}. Finally, \eqref{conv: final l1 - strong convergence of regularization} is a direct consequence of the bound $\frac{1}{m} \|\tem_M\|_{L^\infty(0,T)}^{q-1} \leq C$ and the estimates above.
\end{proof}

Now we pass to the limit in  equation \eqref{eq: approx weak formulation - 2nd equation} for the extra stress tensor $\te$. For this reason fix a test function $\vpsi \in \bochnerlx{q}{\wesym(\Omega)}$. Using Lemma \ref{l:conv} we claim that we can pass to the limit in all terms in equation \eqref{eq: approx weak formulation - 2nd equation}. For example in the convective term 
\eq{\int_{Q_T}\vem\cdot\nabla\tem:\vpsi \dx\dt}
we use strong convergence \eqref{conv: final lp - strong convergence of v^m 2} and weak convergence $\grad\tem \sl \grad\te$ in $L^q(Q_T)$ keeping in mind that $\frac{1}{2\sigma} + \frac{2}{q} \leq 1$ as $q \geq 4$. Exactly the same argument applies also for the terms
\eq{\int_{Q_T}\vem\cdot\div\A\lr{\tem,\vpsi} \dx\dt}
and thus the limit equation is
\begin{equation}\label{eq:limit 2nd equation}
\begin{split}
\int_0^T \duality{\pd{\te}{t}}{\vpsi}_{1,q} \dt + \int_0^T\int_{\Omega} \lr{\ve\cdot\grad\te} : \vpsi \dx\dt  \\ + \eps\int_0^T\int_{\Omega} \U : \grad\vpsi \dx\dt + \int_0^T\int_\Omega \te : \vpsi \dx\dt \\
 = 2\mu_0\int_0^T\int_\Omega \de: \vpsi\dx\dt - \int_0^T\int_\Omega \ve\cdot\div\A(\te,\vpsi)\dx\dt
\end{split}
\end{equation}
for all $\vpsi \in \bochnerlx{q}{\wesym(\Omega)}$. Finally we have to show that
\begin{equation}\label{eq:U je spravne}
\U = \gamma(\nabla\te)\nabla\te.
\end{equation}
To this end we use $\tem$ as a test function in \eqref{eq: approx weak formulation - 2nd equation} and $\te$ as a test function in \eqref{eq:limit 2nd equation}. Note that these are both suitable test functions in the equations. Comparing the results we get
\eq{
\limsup_{m\to\infty} \int_{Q_T} \gamma\lr{\grad\tem}\grad\tem : \grad\tem \dx\dt = \int_{Q_T} \U : \grad\te \dx\dt.
}
Using the monotonicity assumption for $\gamma$ \eqref{assum: monotonicity for gamma} we get
\eq{
0 \leq \int_{Q_T} \lr{ \gamma\lr{\grad\tem}\grad\tem - \gamma\lr{\grad\C}\grad\C} : \lr{\grad\tem - \grad\C} \dx\dt}
and taking the $\limsup$ also
\eq{
0 \leq \int_{Q_T} \lr{ \U - \gamma\lr{\grad\C}\grad\C } : \lr{\grad\te - \grad\C} \dx\dt}
which holds for all $\C \in \bochnerlx{q}{\wesym(\Omega)}$. Thus we can plug in $\C = \te + \beta\E$ for $\beta \in \R$ and any $\E \in \bochnerlx{q}{\wesym(\Omega)}$. Then letting $\beta \sil 0_+$ and $\beta \sil 0_-$ and using the continuity of the function $\gamma$ we finally arrive at
\eq{
0 = \int_{Q_T} \lr{ \U - \gamma\lr{\grad\te}\grad\te } : \E \dx\dt}
which yields \eqref{eq:U je spravne}.

Next we pass to the limit in the momentum equation \eqref{eq: approx weak formulation - 1st equation}. Here we fix a test function $\vphi \in \bochnerlx{\sigma'}{\we{1}{\sigma'}_{0,\div}(\Omega)}$ with $\sigma'$ being a dual exponent to $\sigma \in (1,\frac{5p}{6}\big]$. Using convergences stated in Lemma \ref{l:conv} we can pass to the limit in all terms of \eqref{eq: approx weak formulation - 1st equation} and arrive at
\begin{equation}\label{eq:limit 1st equation incomplete}
\begin{split}
\int_0^T \duality{\pd{\ve}{t}}{\vphi}_{1,p,\div} \dt
- \int_0^T \int_{\Omega} \lr{\ve\otimes\ve} : \grad\vphi \dx\dt \\
+ \int_0^T \int_{\Omega} \S : \grad\vphi \dx\dt = \int_0^T \int_{\Omega} \div \te \cdot \vphi \dx\dt
\end{split}
\end{equation}
for all $\vphi \in \bochnerlx{\sigma'}{\we{1}{\sigma'}_{0,\div}(\Omega)}$. It remains to show that 
\begin{equation}\label{eq:S je spravne}
\S = \mu(\de)\de. 
\end{equation}

\section{Lipschitz truncation and limiting procedure II} \label{s:lip}
 
To prove \eqref{eq:S je spravne}  is not as easy as showing \eqref{eq:U je spravne}. Recall that in  this case $\ve$ is not a suitable test function in the limit equation \eqref{eq:limit 1st equation incomplete}. The rest of the paper is thus devoted to the proof \eqref{eq:S je spravne}.

In order to show the convergence
$$
\mu(\de^m)\de^m \to \mu(\de)\de \qquad \mbox{ a.e. in } Q_T
$$
(i.e. $\S = \mu(\de)\de$), we will use  Theorem 2.16 and Corollary 2.17 from \cite{BrDiSc}. We first introduce certain notation. For $\alpha >0$ we say that $Q = I\times B\subset\er\times \er^3$ is an $\alpha$-parabolic cylinder, if $r_I = \alpha r_B^2$,  where $r_I$ is the radius of the interval $I$ and $r_B$ the radius of the ball $B$. By ${\mathcal Q}^\alpha$ we denote the set of all $\alpha$-parabolic cylinders. For $\kappa>0$ we denote $\kappa Q$ the scaled cylinder $\kappa Q = (\kappa I) \times (\kappa B)$, where $\kappa B$ is the scaled ball with the same center, similarly $\kappa I$. Then $\alpha$-parabolic maximal operators ${\mathcal M}^\alpha$ and ${\mathcal M}^\alpha_s$, $s\in [1,\infty)$ are defined
$$
\begin{array}{c}
\displaystyle ({\mathcal M}^\alpha f)(t,\x) := \sup_{Q' \in {\mathcal Q}^\alpha; (t,\x) \in Q'} \frac{1}{|Q'|} \int_{Q'} |f(s,\y)|\, {\mathrm d}s \, {\mathrm d}\y, \\
\displaystyle ({\mathcal M}^\alpha_s f)(t,\x):= \Big(({\mathcal M}^\alpha |f|^s)(t,\x)\Big)^{\frac 1s}.
\end{array}
$$ 
For $\lambda$, $\alpha>0$ and $\sigma>1$ we define
$$
 {\mathcal O}^\alpha_\lambda (\z) := \Big\{ (t,\x); ({\mathcal M}^\alpha_\sigma(\xi_{\frac 13 Q_0}|\grad^2 \z|)>\lambda \cap ({\mathcal M}^\alpha_\sigma(\xi_{\frac 13 Q_0}|\partial_t \z|)>\lambda\Big\}.
$$
Note that $\z \sim \grad^{-1}\u$; for more precise definition of $\z$ see the proof of Theorem 2.16 in \cite{BrDiSc}. 

We have (see Theorem 2.16 and Corollary 2.17 in \cite{BrDiSc}) 
\begin{thm} \label{4.1}
Let $1<p<\infty$, $p$, $p'>\sigma$. Let $\u_m$ and $\G_m$ satisfy
$$
\langle \partial_t \u_m,\vphi \rangle = \langle\div \G_m,\vphi\rangle
$$
for all $\vphi \in C^\infty_{0,\rm{div}}(Q_0)$, $Q_0= I_0 \times B_0 \subset \er \times \er^3$. Assume that $\u_m$ is a weak null sequence in $L^p(I_0;W^{1,p}(B_0))$ and a strong null sequence in $L^\sigma(Q_0)$ and bounded in $L^\infty(I_0;L^\sigma(B_0))$. Further assume that $\G_m = \G_{1,m} + \G_{2,m}$ such that $\G_{1,m}$ is a weak null sequence in $L^{p'}(Q_0)$ and $\G_{2,m}$ converges strongly to zero in $L^\sigma(Q_0)$. Then there exists a double sequence $\{\lambda_{m,k}\}\subset \er^+$ and $k_0 \in \en$ with  
\begin{itemize}
\item[(a)] $2^{2^k} \leq \lambda_{m,k} \leq 2^{2^{k+1}}$ \newline
such that the double sequence $\u_{m,k} := \u_{\lambda_{m,k}}^{\alpha_{m,k}} \in L^1(Q_0)$, $\alpha_{m,k} := \lambda_{m,k}^{2-p}$ and ${\mathcal O}_{m,k}:={\mathcal O}_{\lambda_{m,k}}^{\alpha_{m,k}}$ defined above satisfy for all $k\geq k_0$
\item[(b)] $\u_{m,k} \in L^s(\frac 14 I_0;W^{1,s}_{0,\rm{div}}(\frac 16 B_0))$ for all $s <\infty$ and $\operatorname{supp}\, \u_{m,k} \subset \frac 16 Q_0$
\item[(c)] $\u_{m,k} = \u_m$ a.e. on $\frac 18 Q_0 \setminus {\mathcal O}_{m,k}$
\item[(d)] $\|\nabla \u_{m,k} \|_{L^\infty(\frac 14 (Q_0))} \leq c \lambda_{m,k}$
\item[(e)] $\u_{m,k} \to 0$ in $L^\infty(\frac 14 Q_0)$ for $m \to \infty$ and $k$ fixed
\item[(f)] $\nabla \u_{m,k} \rightharpoonup^* 0$ in $L^\infty(\frac 14 Q_0)$ for $m \to \infty$ and $k$ fixed
\item[(g)] $\limsup _{m\to \infty} \lambda^p_{m,k} |{\mathcal O}_{m,k}| \leq c 2^{-k}$
\item[(h)] $\limsup_{m\to \infty}\Big| \int_{Q_0}\G_m: \nabla\u_{m,k}\dx\dt\Big| \leq c \lambda_{m,k}^p |{\mathcal O}_{m,k}|$
\item[(i)] Additionally, let $\zeta \in C^\infty_0(\frac 16 Q_0)$ with $\chi_{\frac 18 Q_0} \leq \zeta \leq \chi_{\frac 16 Q_0}$. Let $\u_m$ be uniformly bounded in $L^\infty(I_0;L^\sigma(B_0))$, then for every $\mathbf{K} \in L^{p'}(\frac 16 Q_0)$
$$
\limsup_{m\to \infty} \Big| \Big(\int_{Q_0}(\G_{1,m}+\mathbf{K}):\nabla \u_m\Big)\zeta \chi_{{\mathcal O}_{m,k}^C}\dx\dt\Big| \leq c 2^{-\frac kp}
$$   
\end{itemize}
\end{thm} 

We apply this theorem to our problem \eqref{eq:approx}; cf. Theorem 3.1 in \cite{BrDiSc}. We denote $\u_m = \ve_m - \ve$.  Then
$$
\begin{array}{c}
\u_m \rightharpoonup {\mathbf 0} \quad \mbox{ in } L^p(0,T;W^{1,p}_{0,\textrm{div}}(\Omega)), \\
\u_m \to {\mathbf 0} \quad \mbox{ in }L^{2\sigma}(Q_T), \\
\u_m \rightharpoonup^* {\mathbf 0} \quad \mbox{ in }L^\infty(0,T;L^2(\Omega)),  
\end{array}
$$ 
see \eqref{conv: final lp - weak convergence of v^m}, \eqref{conv: final lp - strong convergence of v^m 2} and \eqref{conv: final lp - weak convergence of v^m 2}. Further
$$
\int_0^T \int_{\Omega} \u_m \cdot \partial_t \vphi \dx\dt= \int_0^T \int_\Omega \G_m :\nabla \vphi\dx\dt
$$
for all $\vphi \in C^\infty_{0,\infty}(Q_T)$, where $\G_m = \G_{1,m}+ \G_{2,m}$ with
$$
\begin{array}{c}
\G_{1,m} = \mu(\de_m)\de_m-\S, \\
\G_{2,m} = -\ve_m \otimes \ve_m \xi(|\ve_m|) + \ve\otimes \ve + \te^m-\te.
\end{array}
$$
We have $\|\G_{1,m}\|_{L^{p'}(0,T;L^{p'}(\Omega))} \leq C$ and $\G_{2,m} \to \mathbf{0}$ in $L^{\sigma_1}(Q_T)$ with $\sigma_1 = \min\{2\sigma,\overline{r}\}$, see \eqref{ineq: a priori estimate}, \eqref{conv: final lp - strong convergence of T^m} and \eqref{conv: final lp - strong convergence of v^m otimes v^m}.

Take now $Q\subset\subset(0,T)\times \Omega$. Due to properties mentioned above, assumptions of Theorem \ref{4.1} (i) are fulfilled. Hence,  plugging in $\mathbf{K} = \S - \mu(\de)\de$ we have for $\zeta \in C^\infty_0(\frac 16 Q_0)$ 
$$
\limsup_{m\to \infty}\Big|\int_0^T \int_{\Omega}
\Big(\G_{1,m} + \S - \mu(\de)\de \Big):\nabla \u_m \zeta \chi_{{\mathcal O}_{m,k}^C}\dx\dt \Big| \leq C 2^{-\frac kp}.
$$
Therefore 
$$
\limsup_{m\to \infty}\Big|\int_0^T \int_{\Omega}
\Big(\mu(\de_m)\de_m - \mu(\de)\de \Big):\nabla \u_m \zeta \chi_{{\mathcal O}_{m,k}^C}\dx\dt \Big| \leq C 2^{-\frac kp}.
$$
Take $\theta \in (0,1)$. By virtue of H\"older's inequality and Theorem \ref{4.1} (g)
\begin{multline*}
\limsup_{m\to \infty}\Big|\int_0^T \int_{\Omega}
\Big((\mu(\de_m) \de_m - \mu(\de)\de ):\nabla \u_m\Big)^\theta \zeta \chi_{{\mathcal O}_{m,k}}\dx\dt \Big| \\
\leq C\limsup_{m\to \infty} |{\mathcal O}_{k,m}|^{1-\theta} \leq C 2^{-(1-\theta)\frac kp}.
\end{multline*}
Thus, both estimates imply
$$
\limsup_{m\to \infty}\Big|\int_0^T \int_{\Omega}
\Big((\mu(\de_m) \de_m - \mu(\de)\de ):\nabla \u_m\Big)^\theta \zeta\dx\dt \Big| \leq   C 2^{-(1-\theta)\frac kp}.
$$
Taking $\lim_{k\to \infty}$ the right-hand side tends to zero. Using now standard approach from \cite{DM} (see also \cite{FrMaSt} or with more details \cite{KrTh}), due to the strict monotonicity of $\mu$, see \eqref{assum: monotonicity for mu}, we get $\mu(\de_m)\de_m \to \mu(\de)\de$ a.e. in $\frac 18 Q_T$ which implies that $\S = \mu(\de)\de$. The proof of the main theorem is finished.

\bigskip

\noindent {Ond\v{r}ej Kreml} \\
Institute of Mathematics  of the Academy of Sciences of the Czech Republic\\ 
\v{Z}itn\'a 25\\
115 65 Praha 1\\
Czech Republic\\
{\tt kreml@math.cas.cz}

\smallskip

\noindent Milan Pokorn\'y \\
Charles University in Prague, Faculty of Mathematics and Physics \\
Mathematical Institute of Charles University \\
Sokolovsk\'a 83 \\
186 75 Praha 8 \\
Czech Republic \\
{\tt pokorny@karlin.mff.cuni.cz}

\smallskip

\noindent Pavel \v Salom \\
Charles University in Prague, Faculty of Mathematics and Physics \\
Department of Mathematics Education\\
Sokolovsk\'a 83 
186 75 Praha 8 \\
Czech Republic \\
{\tt pavel.salom@gmail.com}


\begin{thebibliography}{99}



\bibitem{BrDiSc} Breit, D., Diening, L., Schwarzacher, S.: {\em Solenoidal Lipschitz truncation for parabolic PDE's}, arXiv: 1209.6522 v2 (2013).

\bibitem{CoKl} Constantin, P., Kliegl, M.: {\em Note on global regularity for two-dimensional Oldroyd-B fluids with diffusive stress}, Arch. Ration. Mech. Anal. \textbf{206} (3) (2012), 725--740.

\bibitem{DM} DalMaso, G., Murat, F.: {\em Almost everywhere convergence of gradients of solutions to nonlinear elliptic systems}. Nonlinear Anal. TMA \textbf{31} (3-4) (1998), 405--412. 

\bibitem{DiRuWo} Diening, L., R\r{u}\v zi\v cka, M., Wolf, J.:  {\em Existence of Weak Solutions for Unsteady Motions of Generalized Newtonian Fluids}, Ann. Sc. Norm. Super. Pisa, Cl. Sci. (5) \textbf{9} (2010), 1--46.

\bibitem{FaHiZi} Fang, D., Hieber, M. Zi, R.: {\em Global existence results for Oldroyd-B fluids in exterior domains: the case of non-small coupling parameters}, Math. Ann. {\bf 357} (2013), no. 2, 687--709.

\bibitem{FrMaSt} Frehse, J., M\'alek, J., Steinhauer, M.: {\em On existence result for fluids with shear dependent viscosity – unsteady flows}, In: ``Partial Differential Equations'', W. J\"ager, J. Ne\v cas, O. John, K. Najzar and J. Star\'a (eds.), Chapman and Hall, 2000, 121--129.


\bibitem{KrTh} Kreml, O.: {\em Mathematical analysis of models for viscoelastic fluids}, Ph.D. thesis, Charles University in Prague, Faculty of Mathematics and Physics, (2010).

\bibitem{KrPo}  Kreml, O.; Pokorn\'y: M.: {\em On the local strong solutions for a system describing the flow of a viscoelastic fluid}, Banach Center Publications, Vol. 86: Nonlocal and Abstract Parabolic Equations and their Applications (2009), 196--205.





\bibitem{LiLiZh} F. Lin, C. Liu and P. Zhang, {\it On hydrodynamics of viscoelastic fluids}, Comm. Pure Appl. Math. 58 (2005), 1437--1471.

\bibitem{LiMa} Lions, P.-L., Masmoudi, N.: {\em Global solutions for some Oldroyd models of non-Newtonian flows}, Chinese Ann. Math. Ser. B, \textbf{21} no. 2 (2000), 131--146.

\bibitem{LiWa} C. Liu and N. J. Walkington, {\it An Eulerian description of fluids containing visco-elastic particles}, Arch. Ration. Mech. Anal. 159 (2001) no. 3, 229--252.

\bibitem{SaTh} \v Salom, P.: {\em Mathematical analysis of a regularized model for visco--elastic nonnewtonian fluid}, M.Sc. thesis, Charles University in Prague, Faculty of Mathematics and Physics, (2012).



\end{thebibliography}
\end{document}